\documentclass[a4paper,10pt,hidelinks]{amsart}
\usepackage{graphicx} % Required for inserting images
\usepackage[a4paper, margin=2 cm]{geometry}
\usepackage{hyperref}
\usepackage{amsthm}
\usepackage{thm-restate}
\usepackage{geometry}
\usepackage{amssymb}
\usepackage{amsmath}
\usepackage{cleveref}
\usepackage{comment}
\usepackage{xcolor}
\usepackage{blkarray}
\usepackage{lmodern}
\usepackage[english]{babel}
\usepackage{float}
\usepackage{tikz}
\usetikzlibrary{arrows.meta,arrows}
\usetikzlibrary{arrows,decorations.markings}
\usetikzlibrary{decorations.pathmorphing}
\usepackage[square,sort,comma,numbers]{natbib}

\newtheorem{theorem}{Theorem}[section]

\newtheorem{corollary}[theorem]{Corollary}

\newtheorem*{claim*}{Claim}

\newtheorem{conjecture}[theorem]{Conjecture}

\newtheorem{problem}[theorem]{Problem}

\theoremstyle{definition}

\newtheorem*{definition*}{Definition}

\author
{
Sergey Norin
}

\thanks{Department of Mathematics and Statistics, McGill University, Montr\'{e}al,
Canada}
\thanks{\texttt{sergey.norin@mcgill.ca}}
\thanks{The research of S.N. was supported by an NSERC Discovery grant.}

\author
{
Raphael Steiner
}
\thanks{R.S.: Department of Mathematics, ETH Z\"{u}rich, Switzerland.}
\thanks{\texttt{raphaelmario.steiner@math.ethz.ch}.
\thanks{The research of R.S. was funded by the SNSF Ambizione Grant No. 216071 of the Swiss National Science Foundation.}
}

\date{\today}

\title{Defective coloring of blowups}

\begin{document}
\maketitle

\begin{abstract}
Given a graph $G$ and an integer $d\ge 0$, its \emph{$d$-defective chromatic number} $\chi^d(G)$ is the smallest size of a partition of the vertices into parts inducing subgraphs with maximum degree at most $d$. Guo, Kang and Zwaneveld recently studied the relationship between the $d$-defective chromatic number of the $(d+1)$-fold (clique) blowup $G\boxtimes K_{d+1}$ of a graph $G$ and its ordinary chromatic number, and conjectured that $\chi(G)=\chi^d(G\boxtimes K_{d+1})$ for every graph $G$ and $d\ge 0$. %They also provided several evidences for the correctness of their conjecture. 

\noindent In this note we disprove this conjecture by constructing graphs $G$ of arbitrarily large chromatic number such that $\chi(G)\ge \frac{30}{29}\chi^d(G\boxtimes K_{d+1})$ for infinitely many $d$. On the positive side, we show that the conjecture holds with a constant factor correction, namely $\chi^d(G\boxtimes K_{d+1})\le \chi(G)\le 2\chi^d(G\boxtimes K_{d+1})$ for every graph $G$ and $d\ge 0$.
\end{abstract}

\section{Introduction}
The %famous
\emph{chromatic number} $\chi(G)$ of a graph $G$ is the smallest number $k\ge 1$ such that there exists a \emph{proper $k$-coloring} of $G$, i.e., a mapping $c:V(G)\rightarrow S$ for some color-set $S$ of size $k$ such that $c(u)\neq c(v)$ for every edge $uv$ of $G$. Many important open problems in graph theory seek upper bounds for the chromatic number of a graph given some structural constraints on the graph. In many of these settings however, it is still of interest if we can find a coloring that in some sense is approximately proper. To make this more precise, let us say that an \emph{improper coloring} is simply a mapping $c:V(G)\rightarrow S$ for some finite color set $S$. Maybe the most popular measure of the ``improperness'' of the coloring $c$ is %to look at 
its maximum monochromatic degree (also called \emph{defect}): For an integer $d\ge 0$, we say that $c$ is \emph{$d$-defective} if for every $s\in S$ the subgraph of $G$ induced by the color class of color $s$ has maximum degree at most $d$, i.e., $\Delta(G[c^{-1}(s)])\le d$. In other words, for every $v\in V(G)$ there are at most $d$ neighbors $u\in N_G(v)$ such that $c(u)=c(v)$. By $\chi^d(G)$, we denote the minimum size of a color set $S$ for which a $d$-defective coloring of $G$ exists.

Defective colorings have been extensively studied in the literature, with important connections to a diverse set of areas in structural and geometric graph theory, and to well-known graph coloring conjectures such as Hadwiger's conjecture. Rather than attempting the elusive task of giving a comprehensive overview, we refer to the 70 page survey article by Wood~\cite{wood} for previous work. 

In this note, we shall be concerned with the $d$-defective chromatic number of blowups of graphs by cliques. Concretely, given a graph $G$ and an integer $d\ge 0$, we will consider $d$-defective colorings of the strong product $G\boxtimes K_{d+1}$, which can be equivalently seen as the graph obtained from $G$ by replacing each vertex by a clique of size $d+1$, and by adding all possible edges between the cliques replacing vertices $u$ and $v$ for every edge $uv\in E(G)$. Formally, we will view the strong product $G\boxtimes K_t$ for some $t\ge 1$ as the graph with vertex set $V(G)\times [t]$ such that two vertices $(u,i)$ and $(v,j)$ are adjacent if and only if $u=v$ and $i\neq j$ or $uv\in E(G)$. We remark in passing that Campbell et al.~\cite{campbell} and Esperet and Wood~\cite{esperetwood} recently also studied (improper) colorings of strong products of graphs. However, their results are not directly relevant to the discussion in this paper. Instead, we shall be concerned with a recent conjecture by Guo, Kang and Zwaneveld~\cite{guo}, which we introduce next. Motivated by studying the tightness of Hoffman's classic spectral bound on the chromatic number~\cite{hoffman} and its extension to the $d$-defective chromatic number by Bilu~\cite{bilu}, they were interested in the connections between the $d$-defective chromatic number of $G\boxtimes K_{d+1}$ and the ordinary chromatic number $\chi(G)$ of $G$ and made the following conjecture, which was also posed as an open problem at the 2025 Graph Theory workshop in Oberwolfach. 

\begin{conjecture}[Guo, Kang and Zwaneveld, cf.~Conjecture~2 in~\cite{guo}]\label{con}
For every graph $G$ and every integer $d\ge 0$, we have $\chi(G)=\chi^d(G\boxtimes K_{d+1})$.
\end{conjecture}

By lifting an optimal proper coloring of $G$, one can easily obtain a $d$-defective coloring of $G\boxtimes K_{d+1}$ with $\chi(G)$ colors, and hence $\chi^d(G\boxtimes K_{d+1})\le \chi(G)$ always holds. Thus, at the core of Conjecture~\ref{con} lies the inequality $\chi(G)\le\chi^d(G\boxtimes K_{d+1})$. Guo, Kang and Zwaneveld provided several pieces of evidence supporting this conjecture, for example by proving it for perfect graphs, graphs with chromatic number at most four, and by proving fractional, clustered and spectral analogues of Conjecture~\ref{con}.

Despite these promising signs, in this paper we refute Conjecture~\ref{con} by showing that it fails for graphs of chromatic number at least $30$ and for every $d\ge 2$.

\begin{theorem}\label{thm1}
For every $d\ge 2$ there exists a graph $G$ such that $\chi^d(G\boxtimes K_{d+1})\le k:=2d^3+2d^2+d+3$ but $\chi(G)>k$. 
\end{theorem}
We remark that the above result is tight in terms of $d$, in the sense that Conjecture~\ref{con} holds for $d<2$. This follows from the result of  Guo, Kang and Zwaneveld~\cite[Corollary 20]{guo} on the clustered variant of Conjecture~\ref{con} as the maximum size of a monochromatic component in a $1$-defective coloring  is at most two.

%\begin{proposition}\label{prop}
%Every graph $G$ satisfies $\chi(G)=\chi^1(G\boxtimes K_2)$.
%\end{proposition}
A natural follow-up question is to ask how badly the conjecture fails. For example, is there at least an absolute constant $c>0$ such that for every graph $G$ and every $d\ge 0$ we have $\chi(G)\le c\cdot \chi^d(G\boxtimes K_{d+1})$? In our next result, we show that this is indeed the case, with $c=2$.

\begin{theorem}\label{thm2}
For every graph $G$ and every integer $d\ge 1$, we have $\chi(G)\le 2\cdot \chi^d(G\boxtimes K_{d})\le 2\cdot \chi^d(G\boxtimes K_{d+1})$.
\end{theorem}
It would be interesting to determine the smallest value of $c$ for which the inequality $\chi(G)\le c\cdot \chi^d(G\boxtimes K_{d+1})$ always holds. Theorem~\ref{thm1} for $d=2$ yields one example of a graph showing that $c\ge \frac{30}{29}$ is necessary. 
\begin{problem}
Determine 
$$\sup\left\{\frac{\chi(G)}{\chi^d(G\boxtimes K_{d+1})}\bigg\vert G \text{ is a graph and }d\ge 0\right\}.$$
\end{problem}

Another natural question is whether assuming that either $\chi(G)$ or $d$ is sufficiently large can help to reduce the constant factor in Theorem~\ref{thm2} arbitrarily close to $1$. Unfortunately, by bootstrapping Theorem~\ref{thm1} we can show this is not the case.
\begin{corollary}\label{cor}
    For all positive integers $d\equiv 2\text{ }(\text{mod }3)$ and $k\equiv 0\text{ }(\text{mod }30)$ there exists a graph $G$ with $\chi(G)=k$ and $\chi(G)\ge \frac{30}{29}\cdot \chi^d(G\boxtimes K_{d+1})$.
\end{corollary}
As a final remark, let us mention that the family of graphs from Corollary~\ref{cor} gives new examples of graphs for which the chromatic number is separated from the Hoffman lower bound~\cite{hoffman} (i.e., the quantity $\frac{\lambda_1-\lambda_n}{-\lambda_n}$, where $\lambda_1,\lambda_n$ denote the largest and smallest eigenvalue of the adjacency matrix of $G$, respectively) by a constant factor. Indeed, it was proved by Guo et al. (see the proof of Lemma~11 in~\cite{guo}) that the Hoffman lower bound also forms a lower bound on $\chi^d(G\boxtimes K_{d+1})$ for every $d\ge 0$. Our claim then follows directly from Corollary~\ref{cor}.

\medskip

\paragraph*{\textbf{Notation.}} For an integer $k\ge 1$, we use $[k]=\{1,\ldots,k\}$ as a shorthand for the first $k$ positive integers. Given a graph $G$, we denote by $V(G)$ its vertex set and by $E(G)$ its edge set. We further use $\Delta(G)$ to denote the maximum degree of $G$. For a vertex $v\in V(G)$, we denote by $N_G(v)$ its neighborhood. For a subset of vertices $X\subseteq V(G)$, we denote by $G[X]$ the induced subgraph of $G$ with vertex set $X$. 

\section{Proofs of the results}
In this section, we present the proofs of our results. We start with the disproof of Conjecture~\ref{con}, by establishing Theorem~\ref{thm1}. The proof of Theorem~\ref{thm1} is based on a result of Bohman and Holzman~\cite{bohman} from 2001. To state this result, we need to introduce some additional notation. Given a graph $G$ and an assignment $L:V(G)\rightarrow 2^\mathbb{N}$ of finite sets (referred to as \emph{lists}) to its vertices, an \emph{$L$-coloring} is defined as a proper coloring $c:V(G)\rightarrow \mathbb{N}$ such that each vertex picks a color from its list, that is, $c(v)\in L(v)$ for every $v\in V(G)$. Further, for any vertex $v\in V(G)$ and any color $\alpha\in L(v)$, let us denote by $d_\alpha^L(v)$ its \emph{color degree}, defined as the number of neighbors $u\in N_G(v)$ such that $\alpha\in L(u)$. 

With this notation at hand, we can state the main result of Bohman and Holzman from~\cite{bohman} as follows. Their result disproved a list coloring conjecture earlier made by Reed~\cite{reed}.
\begin{theorem}[Bohman and Holzman~\cite{bohman}]\label{thm:bohman}
For every $d\ge 2$ there exists a graph $F$ with a list-assignment $L$ such that 
\begin{itemize}
    \item $|L(v)|\ge d+1$ for every $v\in V(F)$,
    \item  $d_\alpha^L(v)\le d$ for every $v\in V(F)$ and $\alpha \in L(v)$,
    \item there exists no $L$-coloring of $F$, and
    \item $\left|\bigcup_{v\in V(F)}L(v)\right|=2d^3+2d^2+d+3$.
\end{itemize}
\end{theorem}
Equipped with this result, we can now present the proof of Theorem~\ref{thm1}. 
\begin{proof}[Proof of Theorem~\ref{thm1}]
Let $d\ge 2$ be any given integer, and let $F$ be the graph and $L$ the list assignment as given by Theorem~\ref{thm:bohman}. Set $k:=2d^3+2d^2+d+3=\left|\bigcup_{v\in F}L(v)\right|$. By reducing list sizes if necessary and renaming colors, we may assume w.l.o.g. in the remainder of the proof that $|L(v)|=d+1$ for every $v\in V(F)$ and $\bigcup_{v\in V(F)}{L(v)} \subseteq [k]$. 

We now construct a graph $G$ based on $F$ as follows. The vertex set of $G$ equals $V(G)=V(F)\cup \{v_1,\ldots,v_k\}$, where $v_1,\ldots,v_k\notin V(F)$ are $k$ additional ``new'' vertices. The edge set of $G$ is obtained from the edge set of $F$ by adding the following additional edges: The vertices $\{v_1,\ldots,v_k\}$ are made pairwise adjacent, i.e. form a clique in $G$. Furthermore, for every $u\in V(F)$ and $i\in [k]$, we have $uv_i\in E(G)$ if and only if $i\notin L(u)$. This finishes the description of $G$. We now claim that it satisfies the properties required by the theorem statement, that is, $\chi^d(G\boxtimes K_{d+1})\le k$ and $\chi(G)>k$. 

The first inequality can be easily verified as follows: For each vertex $v\in V(G)$, let $L(v)=\{\alpha_1^v,\ldots,\alpha_{d+1}^v\}$ be an enumeration of the colors in its list, and let us define a mapping $c^d:V(G\boxtimes K_{d+1})\rightarrow [k]$ by setting $c^d(v,i):=\alpha_i^v$ for every $(v,i)\in V(F)\times [d+1]$, as well as $c^d(v_t,i):=t$ for every $t\in [k]$ and $i\in [d+1]$. We claim that $c^d$ is a $d$-defective coloring of $G\boxtimes K_{d+1}$. First of all, note that by our definition of the edges in $G$, for every $t\in [k]$ the vertices in the fiber $\{v_t\}\times [d+1]$ are only connected to the vertices in fibers of the form $\{u\}\times [d+1]$ for some $u$ such that $t\notin L(u)$ or of the form $\{v_{t'}\}\times [d+1]$ for some $t'\neq t$. For both cases, the definition of $c^d$ readily implies that there are no monochromatic edges spanned between the fiber $\{v_t\}\times [d+1]$ and the rest of the graph $G\boxtimes K_{d+1}$. Hence, each fiber $\{v_1\}\times [d+1],\ldots,\{v_k\}\times [d+1]$ spans their own monochromatic component of the coloring $c^d$ of $G\boxtimes K_{d+1}$, and the maximum degree in these components is trivially exactly $d$. Hence, in the following it suffices to verify that the restriction of the coloring $c^d$ to the induced subgraph $F\boxtimes K_{d+1}$ is $d$-defective. To see this, consider any vertex $(v,i)\in V(G\boxtimes K_{d+1})$. Then by definition of $c^d$, the number of neighbors of $(v,i)$ in $G\boxtimes K_{d+1}$ that are assigned the same color as $(v,i)$ by $c^d$, equals exactly the number of neighbors $u$ of $v$ in $G$ such that $L(u)$ contains the color $c^d(v,i)=\alpha_i^v$. In other words, the degree of $(v,i)$ in its color class induced by $c^d$ equals the color degree $d_{\alpha_i^v}^L(v)$, which by our choice of $F$ is bounded from above by $d$. This confirms that %indeed, 
$c^d$ is a $d$-defective coloring of $V(G\boxtimes K_{d+1})$. 

It remains to show that $\chi(G)>k$. Towards a contradiction, suppose there exists a proper $k$-coloring $c:V(G)\rightarrow [k]$ of $G$. Since $\{v_1,\ldots,v_k\}$ form a clique of size $k$ in $G$, possibly after permuting colors we may then assume that $c(v_i)=i$ for every $i\in [k]$. By definition of $G$, every vertex $v\in V(F)$ is adjacent to all the vertices $\{v_i \:|\: i \in [k]\setminus L(v)\}$. Hence, $c(v)$ must be distinct from all colors in $[k]\setminus L(v)$, in other words, we must have $c(v)\in L(v)$ for every $v\in V(F)$. But then the coloring $c$, restricted to $F$, forms an $L$-coloring of $F$, contradicting that we have chosen $F$ such that no $L$-coloring exists. This is the desired contradiction, showing that indeed $\chi(G)>k$. This concludes the proof. 
\end{proof}
With Theorem~\ref{thm1} proved, we next present the deduction of Corollary~\ref{cor} from Theorem~\ref{thm1}.
\begin{proof}[Proof of Corollary~\ref{cor}]
Let $d, k$ be positive integers such that $d\equiv 2\text{ }(\text{mod }3)$ and $k\equiv 0\text{ }(\text{mod }30)$.

By applying Theorem~\ref{thm1} with defect $2$, we find that there exists a graph $G_0$ such that $\chi^2(G_0\boxtimes K_3)\le 29$ and $\chi(G_0)\ge 30$. By removing vertices if necessary, we may assume w.l.o.g. that $\chi(G_0)=30$. Let $m:=\frac{k}{30}$, and let $G'$ be the graph obtained from the disjoint union of $m$ copies $G_1,\ldots,G_m$ of $G_0$ by adding all possible edges between $G_i$ and $G_j$ for all distinct $i,j \in [m]$. It is then easy to see that $\chi(G')=m \cdot \chi(G_0)=30m=k$.

Next we would like to show that $\chi^d(G'\boxtimes K_{d+1})\le 29m$. To that end, fix a $2$-defective coloring $c:V(G_0\boxtimes K_3)\rightarrow [29]$ of $G_0\boxtimes K_3$. Now, let $c':V(G'\boxtimes K_{d+1})\rightarrow [29m]$ be defined as follows: 

For every vertex $(v,i)\in V(G')\times [d+1]$, we set $c'(v,i):=c(v^\ast,i^\ast)+29(t-1)$, where $t\in [m]$ is the unique index such that $v\in V(G_t)$, $v^\ast\in V(G_0)$ denotes the vertex in $G_0$ corresponding to $v$, and $i^\ast\in [3]$ is unique such that $1+\frac{d+1}{3}(i^\ast-1)\le i\le \frac{d+1}{3}i^\ast$.

We now want to bound the maximum monochromatic degree in the coloring $c'$. Note that $c'$ assigns disjoint sets of colors on $V(G_1)\times [d+1],\ldots,V(G_m)\times [d+1]$. Thus, w.l.o.g. it suffices to bound the maximum monochromatic degree in the restriction of $c'$ to $V(G_1)\boxtimes K_{d+1}$. So, consider any vertex $(v,i)\in V(G_1)\boxtimes [d+1]$ and let $(v^\ast, i^\ast)\in V(G_0)\times [3]$ be defined as above. Let $N$ denote the set of vertices in the closed neighborhood of $(v,i)$ in $G_1\boxtimes K_{d+1}$ which gets assigned the same color as $(v,i)$ by $c'$. We can then see from the definition of $c'$ that $(u,j)\in N$ if and only if $(u^\ast,j^\ast)$ is a vertex in the closed neighborhood of $(v^\ast,i^\ast)$ in $G_1\boxtimes K_3$ that gets assigned the same color as $(u^\ast,j^\ast)$ by $c$, where $u^\ast$ is the vertex of $G_0$ corresponding to $u$ and $j^\ast\in [3]$ is such that $1+\frac{d+1}{3}(j^\ast-1)\le j\le \frac{d+1}{3}j^\ast$. By our choice of $c$ there can be in total at most three such vertices $(u^\ast,j^\ast)$, and hence in total at most $3\cdot \frac{d+1}{3}=d+1$ vertices in $N$. Hence, there are at most $|N|-1\le d$ neighbors of $(v,i)$ with the same color as $(v,i)$ under $c'$. This shows that $c'$ is indeed a $d$-defective coloring. Hence, we have $\chi^d(G'\boxtimes K_{d+1})\le 29m$, as desired. This concludes the proof of the corollary.
\end{proof}
We proceed with the proof of our second main result, Theorem~\ref{thm2}. It is based on the following sufficient condition for so-called \emph{independent transversals} due to Haxell~\cite{haxell}. Given a graph $H$ and a partition $(V_1,\ldots,V_\ell)$ of its vertex set, an independent transversal of this partition in $G$ is defined to be a subset $I\subseteq V(G)$ such that $I$ is independent in $G$ and $|I\cap V_i|=1$ for every $i\in [\ell]$. 
\begin{theorem}[cf. Theorem~2 in~\cite{haxell}]\label{thm:haxell}
If $H$ is a graph with a partition $(V_1,\ldots,V_\ell)$ of its vertex set such that $|V_i|\ge 2\cdot \Delta(H)$ for every $i\in [\ell]$, then there exists an independent transversal.
\end{theorem}

Finally, we prove Theorem~\ref{thm2}.

\begin{proof}[Proof of Theorem~\ref{thm2}]
Let $G$ be any given graph and let $d\ge 1$ be an integer. The graph $G\boxtimes K_d$ is isomorphic to a subgraph of $G\boxtimes K_{d+1}$, and so $\chi^d(G\boxtimes K_d)\le \chi^d(G\boxtimes K_{d+1})$ holds trivially and it suffices to show that $\chi(G)\le 2\chi^d(G\boxtimes K_d)$. Let us denote $k:=\chi^d(G\boxtimes K_d)$ and let $c:V(G)\times [d]\rightarrow [k]$ be a $d$-defective coloring of $G\boxtimes K_{d}$. Define an auxiliary graph $H$ as follows: The vertex set is $V(H):=V(G)\times [d]\times \{1,2\}$, and two distinct vertices $(g_1,i,a)$ and $(g_2,j,b)$ of $H$ are connected to each other if and only if $g_1g_2 \in E(G)$, $c(g_1,i)=c(g_2,j)$ and $a=b$. Note first that the maximum degree $\Delta(H)$ is bounded from above by $d$: For every vertex $(g_1,i,a)\in V(H)$, each of its neighbors is of the form $(g_2,j,a)$ where $(g_2,j)\in V(G\boxtimes K_{d})$ is a neighbor of $(g_1,i)$ in $G\boxtimes K_d$ which is assigned the same color as $(g_1,i)$ by the coloring $c$. Since $c$ by assumption was a $d$-defective coloring, the number of such neighbors can be at most $d$. Thus, $\Delta(H)\le d$, as desired.

Next, consider the vertex-partition $(V_g)_{g\in V(G)}$ of $V(H)$ given by the fibers $V_g:=\{g\}\times [d]\times \{1,2\}$ for every $g\in V(G)$. Since $|V_g|=2d\ge 2\Delta(H)$ for every $g\in V(G)$, the conditions of Theorem~\ref{thm:haxell} are met, implying that there exists an independent transversal $I$ for the partition $(V_g)_{g\in V(G)}$ in $H$. By definition, this means that for every $g\in V(G)$ we have $|I \cap V_g|=1$. For each $g\in V(G)$, let us define $c'(g)\in [k]\times \{1,2\}$ as the pair $(c(g,i),a)$, where $(i,a)\in [d]\times \{1,2\}$ is chosen uniquely such that $(g,i,a)\in I$. 

Finally, we claim that the so-defined mapping $c':V(G)\rightarrow [k]\times \{1,2\}$ forms a proper coloring of $G$. Towards a contradiction, suppose that there exists an edge $g_1g_2\in E(G)$ with $c'(g_1)=c'(g_2)$. Then there exist $i,j\in [d]$ and $a, b\in \{1,2\}$ such that $(g_1,i,a), (g_2,j,b)\in I$ as well as $(c(g_1,i),a)=c'(g_1)=c'(g_2)=(c(g_2,j),b)$. Hence, $g_1g_2\in E(G)$, $c(g_1,i)=c(g_2,j)$ and $a=b$. By definition of $H$, this implies that the members $(g_1,i,a)$ and $(g_2,j,b)$ of $I$ are adjacent in $H$, contradicting that $I$ is an independent set in $H$. This shows that $c$ is indeed a proper coloring of $G$, and hence (since it uses at most $2k$ colors), we have $\chi(G)\le 2k=2\chi^d(G\boxtimes K_d)$, as desired. This concludes the proof.
\end{proof}
\begin{comment}

Finally, a small adaption of the method used in the previous proof allows to give a very short proof of Proposition~\ref{prop} promised in the introduction. 
\begin{proof}[Proof of Proposition~\ref{prop}]
We already observed the inequality $\chi(G)\ge \chi^1(G\boxtimes K_2)$ in the introduction, so it suffices to prove that conversely $\chi(G)\le \chi(G\boxtimes K_2)$. Let $c^1:V(G)\times \{1,2\}\rightarrow [k]$ with $k:=\chi^1(G\boxtimes K_2)$ be an optimal $1$-defective coloring of $G\boxtimes K_2$. Let $H$ be the spanning subgraph of $G\boxtimes K_2$ in which an edge between vertices $(u,i), (v,j)$ is kept if and only if $c^1(u,i)=c^1(v,j)$. Since $c$ is a $1$-defective coloring, we have that $\Delta(H)\le 1$. Consider the vertex-partition $(\{g\}\times \{1,2\})_{g\in V(G)}$ of $V(H)$. All parts have size $2\ge 2\Delta(H)$, and hence an application of Theorem~\ref{thm:haxell} yields that there exists an independent set $I$ in $H$ such that $I$ contains exactly one of the two elements $(g,1), (g,2)$ for every $g\in V(G)$. We can now define a proper $k$-coloring $c$ of $G$ by for every $g\in V(G)$ setting $c(g):=c^1(g,i)$ where $i\in\{1,2\}$ is unique such that $(g,i)\in I$. It is readily verified that, since $I$ is independent in $H$, we have that $c$ is a proper $k$-coloring of $G$. Hence, $\chi(G)\le k=\chi^1(G\boxtimes K_2)$, as desired.
\end{proof}
\end{comment}

\paragraph*{\textbf{Acknowledgments.} We would like to gratefully acknowledge Carla Groenland and Freddie Illingworth for discussions on this topic during the 2025 Oberwolfach Graph Theory workshop, which contributed to the results. We also would like to thank Krystal Guo, Ross Kang and Gabri\"{e}lle Zvaneveld for comments on the manuscript. This research was completed at the 2025 Graph Theory Workshop held at the Bellairs Research Institute in Holetown, Barbados.}

\bibliographystyle{abbrvurl}
\bibliography{references}

\end{document}